\documentclass[11pt,reqno]{amsproc}

\title[]{Global weak solutions for SQG in bounded domains}

\author{Peter Constantin}
\address{Department of Mathematics, Princeton University, Princeton, NJ 08544}
\email{const@math.princeton.edu}

\author{Huy Quang Nguyen}
\address{Program in Applied and Computational Mathematics, Princeton University, Princeton, NJ 08544}
\email{qn@math.princeton.edu}

\usepackage[margin=1in]{geometry}
\usepackage{amsmath, amsthm, amssymb}
\usepackage{times}
\usepackage{color}
\usepackage{hyperref}

\newcommand{\bq}{\begin{equation}}
\newcommand{\eq}{\end{equation}}
\newcommand{\bqa}{\begin{eqnarray*}}
\newcommand{\eqa}{\end{eqnarray*}}
\newcommand{\Rr}{\mathbb{R}}
\newcommand{\pa}{\partial}
\newcommand{\la}{\label}
\newcommand{\fr}{\frac}
\newcommand{\na}{\nabla}
\newcommand{\be}{\begin{equation}}
\newcommand{\ee}{\end{equation}}
\newcommand{\ba}{\begin{array}{l}}
\newcommand{\ea}{\end{array}}
\theoremstyle{plain}
\newtheorem{theo}{Theorem}[section]

\newtheorem{lemm}[theo]{Lemma}

\theoremstyle{definition}
\newtheorem{rema}[theo]{Remark}

\DeclareMathOperator{\cnx}{div}

\DeclareSymbolFont{pletters}{OT1}{cmr}{m}{sl}
\DeclareMathSymbol{s}{\mathalpha}{pletters}{`s}

\def\tt{\theta}

\def\D{\Delta}

\def\eps{\varepsilon}

\def\na{\nabla}
\def\la{\left\lvert}

\def\le{\leq}

\def\L#1{\langle #1 \rangle}

\def\mez{\frac{1}{2}}

\def\ra{\right\rvert}

\def\P{\mathbb P}
\def\I{\mathbb I}
\def\L{\Lambda}
\numberwithin{equation}{section}

\date{today}
\begin{document}
\begin{abstract}
We prove existence of global weak $L^2$ solutions of the inviscid SQG equation in bounded domains. 
\end{abstract}


\noindent

\maketitle
\section{Introduction}
The surface quasigeostrophic equation (SQG) of geophysical significance \cite{held} has many similarities with the incompressible Euler equation \cite{cmt}. One difference however has to do with the behavior of the corresponding nonlinearities in rough function spaces: SQG has weak continuity in $L^2$, while the Euler equation does not. The weak continuity is due to a remarkable commutator structure, and this property was used to prove existence of global weak solutions for SQG in the whole space in the thesis of S. Resnick \cite{Res}. The weak continuity was revisited in the periodic case in \cite{CCW}, used in the proof of absence of anomalous dissipation in \cite{ctv} and generalized for equations with more singular constitutive laws in \cite{cccgw}. In this paper we are concerned with the 
issue of weak solutions in bounded domains.  The dissipative critical SQG has global weak solutions \cite{ConIgn} and global interior regularity \cite{ConIgn2}.
In this paper we prove that the inviscid equation has global $L^2$ weak solutions in  bounded domains. The commutator structure is modified by the absence of translation invariance. The commutator estimates from \cite{ConIgn} are used to handle the nonlinearity; additional commutator estimates, based on those in \cite{ConIgn2} are used to handle the ill effects of absence of translation invariance. The proof uses Galerkin approximations based on the eigenfunctions of the Laplacian with homogeneous Dirichlet boundary conditions. The result can also be obtained using a vanishing viscosity approximation.

Let $\Omega\subset \Rr^d$, $d\ge 2$, be an open bounded set with smooth boundary. The inviscid surface quasigeostrophic equation  in $\Omega$ is the equation
\bq\label{SQG}
\partial_t\tt  +u\cdot \nabla\tt =0,
\eq
where $\tt = \tt(x,t)$, $u = u(x,t)$ with $(x, t)\in \Omega\times [0, \infty)$ and with the velocity $u$ given by
\bq\label{u:defi}
u = R_D^\perp \tt :=\nabla^\perp (-\Delta)^{-\mez}\tt.
\eq
Fractional powers of the Laplacian $-\Delta$ are based on eigenfunction expansions of the Laplacian in $\Omega$ with homogeneous Dirichlet boundary conditions. Our main result is:
\begin{theo}\label{main}
Let $\tt_0\in L^2(\Omega)$. There exists a weak solution of \eqref{SQG},  $\tt\in L^\infty([0, \infty); L^2(\Omega))$ with initial data $\tt_0$. That is, for any $T\ge 0$ and $\phi\in C^\infty_0((0, T)\times \Omega)$, $\tt$ satisfies
\bq\label{weak}
\int_0^T\int_\Omega\tt(x, t)\partial_t\phi(x, t)dxdt+\int_0^T\int_\Omega \tt(x, t)u(x, t)\cdot\nabla \phi(x, t)dxdt=0.
\eq
Moreover, $\psi = \L^{-1}\theta\in C([0,\infty); H_0^{{1-\eps}}(\Omega))$ for any  $0<\eps \le 1$ and the initial data is attained
\bq\label{weak:data}
\tt(\cdot, 0)=\tt_0(\cdot)\quad\text{in}~H^{-\eps}(\Omega).
\eq
The Hamiltonian
\bq\label{Ha}
H:=\mez\int_{\Omega}\tt(x, t)\L^{-1}\tt(x, t)dx=\mez\int_{\Omega}\tt_0(x)\Lambda^{-1}\tt_0(x)dx
\eq
is constant in time and, moreover,  $\tt$ obeys the energy inequality
\bq\label{energyin}
\mez\Vert \tt(\cdot, t)\Vert^2_{L^2(\Omega)}\le \mez\Vert \tt_0\Vert^2_{L^2(\Omega)}\quad\text{\it a.e.}~t\ge 0.
\eq
\end{theo}
\begin{rema}
The weak formulation \eqref{weak} means that the SQG equation is satisfied in the sense of distributions. In fact, because
of the boundedness of $R_D$ in $L^2(\Omega)$, the product $\theta u$ is a function,  $\theta u \in L^{\infty}([0,T]; L^1(\Omega))$. The Hamiltonian $H$
is well-defined for allmost all $t\ge 0$ because $\L^{-1}\theta\in H_0^1(\Omega)$. The linear map $\L: H_0^{1-\eps}(\Omega)\to H^{-\eps}(\Omega)$ is continuous, and so $\theta\in C([0,\infty); H^{-\eps}(\Omega))$.
\end{rema}
%%%%%%%%%%
\section{Preliminaries}
Let $\Omega$ be an open bounded set of $\Rr^d$, $d\ge 2$, with smooth boundary. The Laplacian $-\Delta$ is defined on $D(-\Delta)=H^2(\Omega)\cap H^1_0(\Omega)$. Let $\{w_j\}_{j=1}^\infty$ be an orthonormal basis of $L^2(\Omega)$ comprised of $L^2-$normalized eigenfunctions $w_j$ of $-\Delta$, {\it {\it i.e.}}
\[
-\Delta w_j=\lambda_jw_j, \quad \int_{\Omega}w_j^2dx = 1,
\]
with $0<\lambda_1<\lambda_2\le...\le\lambda_j\to \infty$.\\
The fractional Laplacian is defined using eigenfunction expansions,
\[
\Lambda^{s}f\equiv (-\Delta)^{\frac{s}{2}} f:=\sum_{j=1}^\infty\lambda_j^{\frac{s}{2}} f_j w_j\quad\text{with}~f=\sum_{j=1}^\infty f_jw_j,\quad f_j=\int_{\Omega} fw_jdx
\]
for $s\in [0, 2]$ and $f\in {\mathcal {D}}(\Lambda^{s}):=\{f\in L^2(\Omega): \big(\lambda_j^{\frac{s}{2}} f_j\big)\in \ell^2(\mathbb N)\}$. The norm of $f$ in ${\mathcal{D}}(\Lambda^{s })$ is defined by
\[
\Vert f\Vert_{s, D}:=\big(\sum_{j=1}^\infty\lambda_j^s f_j^2\big)^\mez.
\]
It is also well-known that ${\mathcal{D}}(\Lambda)$ and $H^1_0(\Omega)$ are isometric. 
In the language of interpolation theory, 
\[
{\mathcal{D}}(\Lambda^\alpha)=[L^2(\Omega), {\mathcal{D}}(-\Delta)]_{\frac \alpha 2}\quad\forall \alpha\in [0, 2].
\]
As mentioned above,
\[
H^1_0(\Omega)= {\mathcal{D}}(\Lambda)=[L^2(\Omega), {\mathcal{D}}(-\Delta)]_{\mez},
\]
hence
\[
{\mathcal{D}}(\Lambda^\alpha)=[L^2(\Omega), H^1_0(\Omega)]_{\alpha}\quad\forall \alpha\in [0, 1].
\]
Consequently, we can identify ${\mathcal{D}}(\Lambda^\alpha)$ with usual Sobolev spaces (see Chapter 1 \cite{LioMag}):
\bq\label{identify}
{\mathcal{D}}(\Lambda^\alpha)=
\begin{cases}
H^\alpha_0(\Omega) &\quad\text{if}~ \alpha\in [0, 1]\setminus\{\mez\},\\
H^\mez_{00}(\Omega):=\{ u\in H^\mez_0(\Omega): u/\sqrt{d(x)}\in L^2(\Omega)\}&\quad\text{if}~ \alpha=\mez.
\end{cases}
\eq
Here and below $d(x)$ is the distance to the boundary of the domain:
\be
d(x) = d(x, \pa\Omega).
\label{dx}
\ee
The following estimate for the commutator of $\Lambda$ with multiplication by a function was proved in \cite{ConIgn} using the method of harmonic extension:
\begin{theo}[\protect{Theorem 2, \cite{ConIgn}}]\label{Commutator:CI}
Let $\chi\in B(\Omega)$ with $B(\Omega)=W^{2, \infty}(\Omega)\cap W^{1, \infty}(\Omega)$ if $d\ge 3$, and $B(\Omega)=W^{2, p}(\Omega)$ with $p>2$ if $d=2$. There exists a constant $C(d, p, \Omega)$ such that
\[
\Vert [\Lambda, \chi]\psi\Vert_{\mez, D}\le C(d, p, \Omega)\Vert \chi\Vert_{B(\Omega)}\Vert \psi\Vert_{\mez, D}.
\] 
\end{theo}
We also need a pointwise estimate for the commutator of the fractional Laplacian with differentiation.
\begin{theo}\label{Commutator:CN}
For any $p\in [1, \infty]$ and $s\in (0, 2)$ there exists a positive constant $C(d, s, p, \Omega)$ such that
\[
\la [\Lambda^s, \nabla]\psi(x)\ra\le C(d, s, p, \Omega)d(x)^{-s-1-\frac dp}\Vert \psi\Vert_{L^p(\Omega)}
\]
holds for all $x\in \Omega$.
\end{theo}
The proof follows closely the proof for the $p=\infty$ case which was done in \cite{ConIgn2} (see Lemma 6 there) using the heat kernel representation of the fractional Laplacian together with a cancelation of the heat kernel of  $\Rr^d$. We  apply this theorem to the stream function $\psi=\Lambda^{-1}\tt$ which is in $H^1_0(\Omega)$ and thus not necessarily in $L^\infty (\Omega)$. The proof of Theorem \ref{Commutator:CN} is provided in the Appendix.
  %%%%%%%%%%%%%%%%%%%%%%%%%%
\section{Proof of Theorem \ref{main}}
Let $\Omega\subset\Rr^d$ be an open bounded set with smooth boundary. Denote by $\P_m$ the projection in $L^2$ onto the linear span $L^2_m$ of eigenfunctions $\{w_1,...,w_m\}$, {\it {\it i.e.}}
\[
\P_m f=\sum_{j=1}^mf_jw_j\quad\text{for}~f=\sum_{j=1}^\infty f_jw_j.
 \]
Let $\phi\in C_0^{\infty}(\Omega)$ and let $\phi_j = \int_{\Omega}\phi(x) w_j(x)dx$ be the eigenfunction expansion coefficients of $\phi$. Let us note that
\be
|\phi_j| \le C_N\lambda_j^{-N}
\label{phij}
\ee
holds with $C_N$ depending only on  $\phi$ and $N\ge 0$, for $j\ge 1$,
\be
C_N = \Vert\D^N\phi\Vert_{L^2(\Omega)}.
\label{cn}
\ee
This follows from repeated integration by parts using $-\D w_j = \lambda_j w_j$ and Schwartz inequalities. By elliptic regularity estimates, we obtain for all $k\in \mathbb N$ that
\[
\Vert w_j\Vert_{H^k(\Omega)}\le C_k\lambda_j^{\frac k2}.
\]
We know from the easy part of Weyl's asymptotic law that $\lambda_j \ge Cj^{\frac {2}{d}}$. Therefore, with sufficiently large  $N$ satisfying $\frac 2d(N-\fr{k}{2}) > 1$ we deduce that
\[
\Vert (\I - \P_m)\phi\Vert_{H^k(\Omega)}\le \sum_{j=m+1}^\infty |\phi_j|\Vert w_j\Vert_{H^k(\Omega)}\le C_{k, N}\sum_{j=m+1}^\infty \lambda_j^{\frac{k}{2}-N}\to 0
\]
as $m\to \infty$. We proved therefore:
\begin{lemm}\label{idap} Let $\phi\in C_0^{\infty}(\Omega)$. For all $k\in \mathbb N$ we have
\bq\label{error:Pm}
\lim_{m\to\infty}\Vert\left(\mathbb I - \P_m\right)\phi\Vert_{H^k(\Omega)} = 0.
\eq
\end{lemm}
Next, we adapt the well-known commutator representation of the nonlinearity in SQG (\cite{Res}, see also \cite{CCW}, \cite{cccgw}) to take into account the lack of translation invariance of $\L$:
\begin{lemm} \label{commu:key}
Let $\psi \in H^1_0(\Omega)$,  $u = \na^{\perp}\psi$ and $\theta = \L\psi$. Let $\phi\in C_0^{\infty}(\Omega)$ be a test function.   Then 
\bq\label{key}
\int_\Omega \tt u\cdot \nabla \phi dx=\mez \int_\Omega  [\Lambda, \nabla^\perp]\psi\cdot \nabla\phi\psi dx-\mez\int_\Omega  \nabla^\perp\psi\cdot [\Lambda, \nabla\phi] \psi dx
\eq
holds.
\end{lemm}
\begin{proof}
First, we note that 
\[
\int_\Omega \tt u\cdot \nabla \phi dx =\int_{\Omega}\Lambda \psi\nabla^\perp\psi\cdot \nabla \phi dx  = -\int_{\Omega}\psi\nabla^\perp\L\psi\cdot \nabla \phi dx,
\]
where we integrated by parts and used the fact that $\na^{\perp}\cdot\na\phi = 0$. The first and middle terms are well defined because $\theta u\in L^1(\Omega)$ because both $\L\psi$ and $\na^{\perp}\psi$ are in $L^2(\Omega)$. The last term is defined because $\na\phi\cdot \na^{\perp}\L\psi\in H^{-1}(\Omega)$ and $\psi\in H_0^1(\Omega)$. 
Commuting $\nabla^\perp$ with $\Lambda$ and then with $\nabla \phi$ leads to
\[
\begin{aligned}
\int_\Omega \tt u\cdot \nabla \phi dx&=-\int_\Omega   \psi [\nabla^\perp,\Lambda]\psi \cdot \nabla\phi  dx-\int_\Omega   \psi \Lambda \nabla^\perp\psi\cdot \nabla\phi dx\\
&=-\int_\Omega  \psi [\nabla^\perp,\Lambda]\psi\cdot \nabla\phi dx-\int_\Omega  \nabla^\perp\psi\cdot \Lambda(\psi \nabla\phi) dx\\
&=-\int_\Omega  [\nabla^\perp,\Lambda]\psi\cdot \nabla\phi\psi dx-\int_\Omega  \nabla^\perp\psi\cdot [\Lambda, \nabla\phi] \psi dx-\int_\Omega  \nabla^\perp\psi\cdot \nabla \phi \Lambda\psi dx\\
&= -\int_\Omega  [\nabla^\perp,\Lambda]\psi\cdot \nabla\phi\psi dx-\int_\Omega  \nabla^\perp\psi\cdot [\Lambda, \nabla\phi] \psi dx-\int_\Omega  \tt u\cdot \nabla \phi dx.
\end{aligned}
\]
Noticing that the last term on the right-hand side is exactly the negative of the left-hand side, we proved \eqref{key}.
\end{proof}
Let us fix $\tt_0\in L^2(\Omega)$ and a positive time $T$.\\
{\bf Step 1.} (Galerkin approximation)  The $m$th Galerkin approximation of \eqref{SQG} is the following ODE system in the finite dimensional space $\P_mL^2(\Omega) = L^2_m$:
\bq\label{Galerkin}
\begin{cases}
\dot \tt_m+\P_m(u_m\cdot\nabla\tt_m)=0&\quad t>0,\\
\tt_m=P_m\tt_0&\quad t=0
\end{cases}
\eq
with $\tt_m(x, t)=\sum_{j=1}^m\tt_ j^{(m)} (t)w_j(x)$ and $u_m={R_D}^\perp\tt_m$ automatically satisfying $\cnx u_m=0$. Note that in general $u_m\notin L^2_m$.
The existence of solutions of \eqref{Galerkin} at fixed $m$ follows from the fact that this is an ODE:
\[
\frac{d\theta^{(m)}_l}{dt} + \sum_{j,k=1}^m\gamma^{(m)}_{jkl}\theta^{(m)}_j\theta^{(m)}_{k} = 0
\label{galmode}
\]
with
\[
\gamma^{(m)}_{jkl} = \lambda_j^{-\frac{1}{2}}\int_{\Omega}\left(\na^{\perp}w_j\cdot\na w_k\right)w_ldx.
\]
Since $\P_m$ are self-adjoint in $L^2$ , $u_m$ are divergence-free and $w_j$ vanish at the boundary $\partial\Omega$, an integration by parts gives
\[
\int_\Omega \tt_m\P_m(u_m\cdot\nabla \tt_m)dx=\int_\Omega \tt_m u_m\cdot\nabla \tt_mdx=0\quad\forall m\in \mathbb N.
\]
 It follows that $\mez\frac{d}{dt}\Vert \tt_m(\cdot, t)\Vert^2_{L^2}=0$ and thus for all $t\in [0, T]$
\bq\label{L^2bound}
\mez\Vert \tt_m(\cdot, t)\Vert^2_{L^2(\Omega)}=\mez\Vert \P_m\tt_0(\cdot, 0)\Vert^2_{L^2(\Omega)}\le \mez \Vert \tt_0\Vert^2_{L^2(\Omega)}.
\eq
This can be seen directly on the ODE because $\gamma^{(m)}_{jkl}$ is antisymmetric in $k,l$.
Therefore, the smooth solution $\tt_m$ of \eqref{Galerkin} exists globally and obeys the $L^2$ bound \eqref{L^2bound}. Let $\phi\in C^\infty_0((0, T)\times \Omega)$ be a test function.  Integrating by parts we obtain 
\bq\label{weak:approx} 
\int_0^T\int_\Omega\tt_m(x, t)\partial_t\phi(x, t) dxdt+\int_0^T \int_\Omega \tt_m(x, t)u_m(x, t)\cdot\nabla \P_m\phi(x, t)dxdt=0.
\eq
Let us denote
\begin{equation}
\psi_m =\L^{-1}\tt_m\in L^2_m.
\label{psim}
\end{equation}

We also have
\[
\int_{\Omega}\psi_m\P_m(u_m\cdot \nabla \tt_m)dx=\int_{\Omega}\psi_m \cnx(\na^\perp\psi_m \tt_m)dx=-\int_{\Omega}\na \psi_m \cdot\na^\perp\psi_m\tt_mdx=0,
\]
and therefore
\bq\label{Hm}
\int_{\Omega}\psi_m(x, t)\tt_m(x, t)dx=\int_{\Omega}\psi_m(x, 0)\tt_m(x, 0)dx\quad\forall t\ge 0,~m\in \mathbb N.
\eq
{\bf Step 2.} (Weak and strong convergences). In view of \eqref{L^2bound} the sequence $\tt_m$ is uniformly in $m$ bounded in $L^\infty([0, T]; L^2(\Omega))$ and consequently the same is true for $u_m = R_D^\perp \tt_m$. The sequence $\psi_m = \L^{-1}\theta_m$ is uniformly bounded in $L^{\infty}([0,T]; H^1_0(\Omega))$. In addition, the sequence $\pa_t \psi_m$ is bounded in $L^{\infty}([0,T]; H^{-r}(\Omega))$ for $r> \fr{3d}{2}$. Indeed, from the equation (\ref{Galerkin}) we have that
\be
\pa_t\psi_m = -\L^{-1}\P_m(\na\cdot(u_m\theta_m)) =  \P_m (R_D)^*\cdot(u_m\theta_m)
\label{dtpsim}
\ee
because $\L^{-1}$ and $\P_m$ commute, and the $L^2(\Omega)$ formal adjoint of $R_D$ is $R_D^* = -\L^{-1}\na$.
Testing with a test function $\phi$ we have
\[
\int_{\Omega}\pa_t\psi_m \phi dx = \int_{\Omega}(u_m\theta_m)\cdot R_D(\P_m\phi)dx
\]
and by taking $\phi\in H_0^r(\Omega)$ we made sure that $\P_m\phi$ is uniformly in $m$  bounded in $X_{\alpha}(\Omega) =\{p\in C^{\alpha}(\overline{\Omega}),\;\; p_{\left |\right. \pa\Omega}= 0\}$. Indeed, the expansion coefficients $\phi_j$ decay as in (\ref{phij}), (\ref{cn}), and choosing $N=\fr{r}{2}>\fr{k+d}{2}$, $k>\fr{d}{2}$ ensures the uniform bound of $\P_m\phi$ in $H^k(\Omega)\cap H_0^1(\Omega)\subset X_{\alpha}(\Omega)$.
Now it is known that $R_D$ maps continuously $X_{\alpha}(\Omega)$ to
$L^{\infty}(\Omega)$ (and better, \cite{CabTan}). Therefore, from the uniform bound on $u_m\theta_m$ in $L^{\infty}([0,T], L^1(\Omega))$ it follows that
\be
\left|\int_{\Omega}\pa_t\psi_m \phi dx\right| \le C\|\tt_0\|_{L^2(\Omega)}^2\|\phi\|_{L^{1}([0,T]; H_0^r(\Omega))}.
\label{dtpsimb}
\ee
In view of the compact embedding $H_0^1(\Omega)\subset H_0^{1-\epsilon}(\Omega)$  we may use the Aubin-Lions lemma \cite{Lions} with spaces $L^2([0,T]; H_0^1(\Omega))$ and $L^2([0,T]; H^{-r}(\Omega))$ to extract a subsequence of $\psi_m$ which converges weakly in $L^2([0,T]; H_0^1(\Omega))$ to a function $\psi$ and such that the convergence is strong in $C([0,T]; H_0^{1-\eps}(\Omega))$ for $\eps\in(0,1]$. By lower semicontinuity we have also that $\psi\in L^{\infty}([0,T]; H_0^{1}(\Omega))$. The function $\theta = \L\psi$ is then the weak limit in $L^2([0,T]; L^2(\Omega))$ of the sequence $\theta_m$ and the strong limit in $C([0,T]; H^{-\eps}(\Omega))$. The function $\theta$ belongs to $L^{\infty}([0,T]; L^2(\Omega))$.

{\bf Step 3.} (Passage to limit)  Let the test function $\phi\in C^\infty_0((0, T)\times \Omega)$ be fixed. We first apply \eqref{error:Pm} (uniformly in $t$) and Sobolev embedding to deduce
\[
\lim_{m\to \infty}\Vert \na(\I-\P_m)\phi\Vert_{L^\infty([0, T]\times \Omega)}=0,
\]
and hence the difference
\bq\label{errorPm}
\int_0^T\int_\Omega \tt_m(x, t){R_D}^\perp\tt_m(x, t)\cdot\nabla \P_m\phi(x, t)dxdt-\int_0^T\int_\Omega \tt_m(x, t){R_D}^\perp\tt_m(x, t)\cdot\nabla \phi(x, t)dxdt
\eq
converges to $0$ as $m\to \infty$. Next, using Lemma \ref{commu:key} we write 
\[
\begin{aligned}
I_m&:=\int_0^T\int_\Omega \tt_m(x, t)u_m(x, t)\cdot\nabla \phi(x, t)dxdt-\int_0^T\int_\Omega \tt(x, t)u(x, t)\cdot\nabla \phi(x, t)dxdt\\
&=\mez\int_0^T\int_\Omega  [\Lambda, \nabla^\perp](\psi_m-\psi)\cdot \nabla \phi\psi dxdt+\mez \int_0^T\int_\Omega  [\Lambda, \nabla^\perp]\psi_m\cdot \nabla \phi(\psi_m-\psi) dxdt\\
&\quad-\mez\int_0^T\int_\Omega  \nabla^\perp(\psi_m-\psi)\cdot [\Lambda, \nabla\phi] \psi dxdt-\mez\int_0^T\int_\Omega  \nabla^\perp\psi_m\cdot [\Lambda, \nabla \phi] (\psi_m-\psi) dxdt\\
&=:\mez\sum_{j=1}^4 I_m^j.
\end{aligned}
\]
According to Theorem \ref{Commutator:CN},
\[
\la [\Lambda, \nabla^\perp](\psi_m(x,t)-\psi(x, t))\ra\le C_1d(x)^{-3}\Vert \psi_m(t)-\psi(t)\Vert_{L^2(\Omega)}\le C_2\Vert \psi_m(t)-\psi(t)\Vert_{L^2(\Omega)}
\]
on the support of $\nabla \phi$ which stays away from the boundary. By virtue of the strong convergence of $\psi_m$ in $L^2([0,T]; L^2(\Omega))$ we deduce that
\[
\la I^1_m\ra\le C\Vert \psi_m-\psi\Vert_{L^2([0, T]; L^2(\Omega))}\Vert \nabla\phi\Vert_{L^\infty([0, T]; L^{\infty}(\Omega))}\Vert \psi\Vert_{L^2([0, T]; L^2(\Omega))}\to 0
\]
as $m\to\infty$. The same argument leads to $I_m^2\to 0$. Next, because of of Theorem \ref{Commutator:CI}, $[\Lambda, \nabla\phi] \psi \in L^2([0,T]; {\mathcal{D}}(\Lambda^\mez))\subset L^2([0, T]; H_0^1(\Omega))$ which combined with the fact that $\nabla^\perp(\psi_m-\psi)\rightharpoonup 0$ weakly in $L^2([0, T]; L^2(\Omega))$, implies that $I_m^3\to 0$. Regarding $I_m^4$ we apply Theorem \ref{Commutator:CI} to have 
\[
\la I_m^4\ra\le C\Vert \nabla \psi_m\Vert_{L^2([0, T]; L^2(\Omega))}\Vert \na\phi\Vert_{L^\infty ([0, T];B(\Omega))}\Vert \psi_m-\psi\Vert_{L^2([0, T]; \mez, D)}.
\]
 In view of (\ref{identify}),  $\psi_m\to \psi$ in $L^2([0, T]; \mez, D)$. Consequently, $I_m^4\to 0$ and thus $I_m\to 0$. Sending $m$ to $\infty$ in \eqref{weak:approx} and taking \eqref{errorPm} into account, we obtain \eqref{weak}. Moreover, because of the strong continuity of $\theta $ in $H^{-\epsilon}$ the initial data is attained
\[
\tt_0(\cdot, 0)=\lim_{m\to \infty}\tt_m(\cdot, 0)=\lim_{m\to \infty}\P_m\tt_0(\cdot, 0)=\tt_0(\cdot, 0)\quad\text{in}~H^{-\eps},
\]
where the third equality actually holds in $L^2$. The conservation in time of the Hamiltonian follows from the constancy in time of $H_m(t)=\int_{\Omega}\psi_m \L\psi_m dx = \|\psi_m\|^2_{\mez, D}$ (\ref{Hm}). From strong convergence 
of $\psi_m$ to $\psi$ in $C([0,T]; {\mathcal D}(\L^{\mez}))\subset C([0,T]; H^{\mez}_0(\Omega))$ it follows that $H(t) = \|\psi\|^2_{\mez, D}$ is constant in time. Finally, the energy inequality \eqref{energyin} follows from \eqref{L^2bound} and lower semicontinuity.

\section*{Appendix: Proof of Theorem \ref{Commutator:CN}}\label{appendix}
In view of the identity 
\[
\lambda^{\frac s2}=c_s\int_0^\infty t^{-1-\frac s2}(1-e^{-t\lambda})dt
\] 
with $0< s<2$ and 
\[
1=c_s\int_0^\infty t^{-1-\frac s2}(1-e^{-t})dt
\]
we have the representation of the fractional Laplacian via heat kernel:
\bq\label{frac}
\Lambda^s\psi(x)=c_s\int_0^\infty t^{-1-\frac s2}(1-e^{t\Delta})\psi(x)dt,\quad 0<s<2.
\eq
Let $H(x, y, t)$ denote the heat kernel of $\Omega$, {\it {\it i.e.}}
\[
e^{t\Delta} \psi(x)=\int_\Omega H(x, y, t)\psi(y)dy\quad\forall x\in \Omega.
\]
We have from \cite{LiYau} the following bounds on $H$ and its gradient:
\bq\label{H}
ct^{-\frac d2}e^{-\frac{|x-y|^2}{kt}}\le H(x, y, t)\le Ct^{-\frac d2}e^{-\frac{|x-y|^2}{Kt}},
\eq
\bq\label{gradH}
|\nabla_x H(x, y, t)|\le Ct^{-\mez-\frac d2}e^{-\frac{|x-y|^2}{Kt}}
\eq
for all $(x,y)\in \Omega\times \Omega$ and $t>0$. In view of the expansion
\[
H(x, y, t)=\sum_{j=1}^\infty e^{-\lambda_jt}w_j(x)w_j(y)
\]
it is easily seen that $|\nabla_y H(x, y, t)|$ also obeys the bound \eqref{gradH}.\\
Using \eqref{frac} and integration by parts we arrive at
\bq\label{commute}
[\Lambda^s, \nabla]\psi(x)=c_s\int_0^\infty t^{-1-\frac s2}\int_\Omega (\nabla_x+\nabla_y)H(x, y, t) \psi(y)dydt.
\eq
Let $p\in (1, \infty]$ and $\frac 1q=1-\frac 1p$. We have
\bq\label{bound1}
\la [\Lambda^s, \nabla]\psi(x)\ra\le c_s\Vert \psi\Vert_{L^p}\int_0^\infty t^{-1-\frac s2}\left[\int_\Omega \la (\nabla_x+\nabla_y)H(x, y, t)\ra^q dy\right]^{\frac{1}{q}}dt.
\eq
The problem reduces thus to estimating the $L^q$-norm of $(\nabla_x+\nabla_y)H(x, \cdot, t)$. We distinguish two regions of $y$: $|x-y|\ge \frac{d(x)}{10}$ and $|x-y|\le \frac{d(x)}{10}$. We use the elementary estimate
\bq\label{integral}
\int_0^\infty t^{-1-\frac m2}e^{-\frac{p^2}{Kt}}dt\le C_{K, m}p^{-m},\quad m, p, K>0.
\eq
If $|x-y|\ge \frac{d(x)}{10}$, the gradient bound \eqref{gradH} implies
\[
\la(\nabla_x+\nabla_y)H(x, y, t)\ra \le Ct^{-\mez-\frac d2}e^{-\frac{d(x)^2}{200Kt}}e^{-\frac{|x-y|^2}{2Kt}}\quad \forall t>0,
\]
hence, in view of \eqref{integral},
\[
\begin{aligned}
&\int_0^\infty t^{-1-\frac s2}\left[\int_{|x-y|\ge \frac{d(x)}{10}}\la (\nabla_x+\nabla_y)H(x, y, t)\ra^q dy\right]^{\frac{1}{q}}dt\\
&\le C_1\int_0^\infty t^{-1-\frac s2-\mez-\frac d2}e^{-\frac{d(x)^2}{200t}}dt\left[\int_{|x-y|\ge \frac{d(x)}{10}} e^{-\frac{q|x-y|^2}{2Kt}}dy\right]^{\frac{1}{q}}\\
&\le C_2\int_0^\infty t^{-1-\frac s2-\mez-\frac d2+\frac{d}{2q}}e^{-\frac{d(x)^2}{200t}}dt\\
&\le Cd(x)^{-s-1-d+\frac dq}.
\end{aligned}
\]
On the other hand, if $|x-y|\le \frac{d(x)}{10}$ we have from Appendix 1 of \cite{ConIgn2} 
\bq\label{cancel}
\la(\nabla_x+\nabla_y)H(x, y, t)\ra\le Ct^{-\mez-\frac d2}e^{-\frac{d(x)^2}{Ct}},\quad t\le d(x)^2.
\eq
Note that in $\Rr^d$, $(\nabla_x+\nabla_y)H$ vanishes identically. \eqref{cancel} thus reflects the fact that translation invariance is remembered in the solution of the heat equation with Dirichlet boundary data for short time, away from the boundary. The bound \eqref{integral} then yields 
\[
\begin{aligned}
&\int_0^{d(x)^2} t^{-1-\frac s2}\left[\int_{|x-y|\le \frac{d(x)}{10}}\la (\nabla_x+\nabla_y)H(x, y, t)\ra^q dy\right]^{\frac 1q}dt\\
&\le C_1\int_0^{d(x)^2}t^{-1-\frac s2-\mez-\frac d2}d(x)^{\frac dq}e^{-\frac{d(x)^2}{C_2t}}dt\\
&\le C_3\int_0^{d(x)^2}t^{-1-\frac s2-\mez-\frac d2+\frac{d}{2q}}e^{-\frac{d(x)^2}{C_4t}}dt\\
&\le Cd(x)^{-s-1-d+\frac{d}{q}}.
\end{aligned}
\]
To obtain the bound for $[\Lambda^s, \nabla]\psi(x)$, it remains to estimate 
\[
I=\int_{d(x)^2}^\infty t^{-1-\frac s2}\left[\int_{|x-y|\le \frac{d(x)}{10}}\la (\nabla_x+\nabla_y)H(x, y, t)\ra^q dy\right]^{\frac 1q}dt.
\]
Using the gradient bound \eqref{gradH} we have
\[
\begin{aligned}
I&\le C_1\int_{d(x)^2}^\infty t^{-1-\frac s2-\mez-\frac d2}\left[\int_{|x-y|\le \frac{d(x)}{10}}e^{-q\frac{|x-y|^2}{Kt}} dy\right]^{\frac 1q}dt\\
&\le C_2\int_{d(x)^2}^\infty t^{-1-\frac s2-\mez-\frac d2+\frac{d}{2q}}dt\\
&\le Cd(x)^{-s-1-d+\frac{d}{q}}.
\end{aligned}
\]
Putting the above considerations together we arrive at the pointwise estimate
\[
\la [\Lambda^s, \nabla]\psi(x)\ra\le Cd(x)^{-s-1-\frac dp}\Vert \psi\Vert_{L^p}
\]
for all $p\in (1, \infty]$. The case $p=1$ can be proved along the same lines.

{\bf{Acknowledgment.}} The work of PC was partially supported by  NSF grant DMS-1209394

\end{document}